\documentclass[10pt, a4paper]{article}
\usepackage[T1]{fontenc}
\usepackage[utf8]{inputenc}
\usepackage[UKenglish]{babel}
\usepackage[a4paper]{geometry}
\usepackage[inline]{enumitem}
\usepackage{amssymb, amsmath, amsthm}
\usepackage[textsize=scriptsize, bordercolor=lightgray, linecolor=lightgray, backgroundcolor=white]{todonotes}
\usepackage{framed}
\usepackage{csquotes}
\usepackage[backend=biber, style=alphabetic, doi=false, url=false, isbn=false]{biblatex}
\usepackage{bookmark}
\hypersetup{colorlinks=true, linkcolor=blue, citecolor=blue}

\newtheorem*{lemma*}{Key Lemma}
\newtheorem{theorem}{Theorem}
\newtheorem*{corollary*}{Corollary}
\newtheorem*{conjecture*}{Conjecture}
\newtheorem*{definition*}{Definition}
\newtheorem*{notation*}{Notation}
\newtheorem*{question*}{Question}
\newtheorem*{result*}{Main results}
\newtheorem{fact}{Fact}
\newtheorem{proposition}{Proposition}[subsection]

\theoremstyle{definition}
\newtheorem*{remarks*}{Remarks}

\newcommand{\GA}{\operatorname{GA}}
\newcommand{\PGL}{\operatorname{PGL}}
\newcommand{\SO}{\operatorname{SO}}
\newcommand{\SL}{\operatorname{SL}}
\newcommand{\bF}{\mathbb{F}}
\newcommand{\bJ}{\mathbb{J}}
\newcommand{\bK}{\mathbb{K}}
\newcommand{\bR}{\mathbb{R}}

\title{On quasi-Frobenius pairs of finite Morley rank}
\author{Tuna Alt{\i}nel, Luis Jaime Corredor and Adrien Deloro}

\begin{document}

\maketitle

\abstract{We clarify quasi-Frobenius configurations of finite Morley rank.
\begin{enumerate*}
\item
We remove one assumption in an identification theorem by Zamour while simplifying the proof.
\item
We show that a strongly embedded quasi-Frobenius configuration of odd type, is actually Frobenius.
\item
For dihedral configurations, one has $\dim G = 3 \dim C$.
\end{enumerate*}
These results
rely on an interesting phenomenon of closure of non-generic matter under taking centralisers.
}\bigskip

\normalsize

{\centering
\S~\ref{S:introduction}.~Introduction
\quad --- \quad
\S~\ref{S:lemma}.~The Key Lemma
\quad --- \quad
\S~\ref{S:proofs}.~The proofs
\par}
\normalsize

\section{Introduction}\label{S:introduction}

For the hasty reader, our results are the following.

\begin{result*}
Let $(C < G)$ be a definable, connected, quasi-Frobenius pair of finite Morley rank. Suppose that $G$ is $U_2^\perp$ with involutions. 
\begin{itemize}
\item
If $C$ is soluble but $G$ is not, then $C$ is a Borel subgroup or $(C < G) \simeq (\bK^\times < \PGL_2(\bK))$ (Theorem~\ref{t:CoBo}).
\item
The index $[N_G(C):C]$ is either $1$ or $2$, the Pr\"ufer $2$-rank is $1$, and $C = C_G^\circ(i)$ for its unique involution (Theorem~\ref{t:1or2}).
\item
If $[N_G(C):C)] = 2$, then $\dim G = 3 \dim C$ (Theorem~\ref{t:G3C}).
\end{itemize}
\end{result*}

(Theorem~\ref{t:CoBo} improves work by Zamour.)
We also have a question to experts of finite group theory at the end of the introduction.
We now explain and discuss.

\paragraph{Background.}

The present work uses extremely little from the general theory of groups of finite Morley rank: the definition, basic computations on the rank (denoted $\dim$), connected components and the Morley degree $\deg$, the descending chain condition, and some familiarity with semisimple torsion will suffice. In particular there is a notion of a large set: a definable subset $X \subseteq G$ is \emph{generic} if $\dim X = \dim G$. Importantly, two generic subsets of a connected group must intersect. \cite{CGood} contains enough background. We add the definition of a \emph{Borel subgroup}: a definable, connected, soluble subgroup maximal as such.
As always, $U_2^\perp$ means: no infinite elementary abelian $2$-groups. So $U_2^\perp$ with involutions is our way to say `of odd type'. Recall that in that case, Sylow $2$-subgroups are finite extensions of $2$-tori, the latter being $C_{2^\infty}^d$, where $C_{2^\infty}$ is the usual Pr\"ufer quasi-cyclic $2$-group and for an integer $d$ called the \emph{Pr\"ufer $2$-rank} of $G$. Last, if an involution $i$ acts on a definable, connected subgroup $H$ with $C_H^\circ(i) = 1$, then $i$ inverts $H$ and $H$ is abelian \cite[ex.~12 p.~78]{BNGroups}.

\paragraph{Quasi-Frobenius pairs.}

Quasi-Frobenius pairs were introduced and first studied in \cite{DWGeometry} but they received a proper name only in \cite{ZQuasi}.

\begin{definition*}
A pair of groups $(C < G)$ is \emph{quasi-Frobenius} if the following holds:
\begin{itemize}
\item
$C$ is quasi self-normalising in $G$, viz.~$[N_G(C):C] < \infty$;
\item
$C$ has trivial intersections with proper conjugates (for short, `$C$ is \textsc{ti}'), viz.~$(\forall g \in G)(C^g = C\;\vee\; C^g \cap C = 1)$.
\end{itemize}
\end{definition*}

This extends the classical notion of a \emph{Frobenius group} (which could more accurately be called a Frobenius pair), where $C$ is \textsc{ti} with $N_G(C) = C$. The following pairs are quasi-Frobenius:
\begin{itemize}
\item
$(\bK^\times < \PGL_2(\bK))$ for algebraically closed $\bK$;
\item
$(\SO_2(\bR) < \SO_3(\bR))$ for real closed $\bR$.
\end{itemize}

The study of non-Frobenius, quasi-Frobenius pairs therefore lies at the heart of geometric algebra and should be undertaken seriously, at least under model-theoretic assumptions. In the present paper we focus on finite Morley rank. A complete study in the $o$-minimal case is planned in \cite{DOominimal}. The analogue definition for finite groups would replace $[N_G(C):C] < \infty$ by $[N_G(C):C] = 2$, which is fully justified by our Theorem~\ref{t:1or2}; \emph{this natural condition seems not to have received the attention it deserves}. See the end of the introduction. We return to model theory.

We say that a pair $(C < G)$ is \emph{definable}, resp.~\emph{connected}, if both $G$ and $C$ are.

\begin{remarks*}\leavevmode
\begin{itemize}
\item
As opposed to the Frobenius case \cite[Proposition~11.19]{BNGroups}, there is no `automatic definability' for $C$ in the quasi-Frobenius setting.
Eg.~let $G$ be an infinite vector space over $\bF_p$, treated as a pure group, and $C$ be a hyperplane. Then $C$ is not definable but the pair is quasi-Frobenius.
\item
However, if $G$ is connected and $C$ is definable, then $C$ is connected.
Otherwise $C^\circ$ and $\check{C} = C\setminus C^\circ$ are two \textsc{ti} subsets; however $N_G(C^\circ) = N_G(C) = N_G(\check{C})$, so by a quick dimension computation, both $\bigcup_{g \in G}(C^\circ)^g$ and $\bigcup_{g \in G} \check{C}^g$ are generic. Therefore they meet: a contradiction.
\end{itemize}
\end{remarks*}

\paragraph{State of the art.}
A strong division line separates two classes of quasi-Frobenius configurations. A pair $(H < G)$ is \emph{strongly embedded} if $H$ contains involutions, but $H \cap H^g$ never does for $g \notin N_G(H)$.

\begin{fact}[{\cite[Proposition~1]{DWGeometry}}]\label{f:DW}
Let $(C < G)$ be a definable, connected, quasi-Frobenius pair of finite Morley rank. Suppose that $G$ is $U_2^\perp$ with involutions.
Then:
\begin{itemize}
\item
either $[N_G(C) : C] = 2$, $C$ is abelian, $N_G(C)\setminus C$ is a set of involutions inverting $C$, the Pr\"ufer $2$-rank is $1$, and a lot more is known (`dihedral configuration');
\item
or $[N_G(C):C]$ is odd (possibly $1$) and $(N_G(C) < G)$ is strongly embedded.
\end{itemize}
\end{fact}

Throughout, `dihedral' and `strongly embedded' will always refer to this main dichotomy.
Ideally one would hope strongly embedded configurations to be soluble, but this seems out of reach.
It is however expected that the dihedral case should come from the algebraic world.

\begin{conjecture*}[`$A_1$ conjecture' from \cite{DWGeometry}]
Let $(C < G)$ be a definable, connected, quasi-Frobenius pair of finite Morley rank. Suppose that $G$ is $U_2^\perp$ with involutions. Suppose the configuration is dihedral. Then $G \simeq \PGL_2(\bK)$.
\end{conjecture*}

This conjecture was devised to remove one of the possible pathologies surviving \cite{DJInvolutive}, namely CiBo$_2$, which stands for `\textbf{C}entraliser of an \textbf{i}nvolution is a \textbf{Bo}rel subgroup with Weyl group of order \textbf{2}'. 
A favourable indication towards $A_1$ was obtained by Zamour, who however phrased it quite differently from the following.

\begin{fact}[{\cite[Th\'{e}or\`{e}me~1.0.6]{ZQuasi}}]\label{f:Zamour}
Let $(C < G)$ be a definable, connected, quasi-Frobenius pair of finite Morley rank. Suppose that $G$ is $U_2^\perp$ with involutions. Suppose further that $C$ is soluble but $G$ is not. \emph{Last, suppose that the configuration is dihedral.}
Then:
\begin{itemize}
\item
either $C$ is a Borel subgroup;
\item
or $(C < G) \simeq (\bK^\times < \PGL_2(\bK))$.
\end{itemize}
\end{fact}

The striking analogy between Fact~\ref{f:Zamour} in this form and the series \cite{CJTame, DGroupes1, DJInvolutive} suggests that \emph{quasi-Frobenius configurations of finite Morley rank behave like $N_\circ^\circ$-groups}.
Zamour has replaced a global smallness assumption by a local, structural one; he substituted the use of unipotence in \cite[Proposition~3]{DJInvolutive} with his classification of soluble, quasi-Frobenius pairs (Fact~\ref{f:soluble} below). Hence \emph{ad hoc} reasoning on an elementary Jordan decomposition entirely substitutes Burdges' version of Bender's local analysis. This is demonstrated in the key lemma (\S~\ref{S:lemma}).

Fully exploiting the analogy dictates what to prove next: remove assumptions on the Weyl group, study strongly embedded configurations, and get dimension estimates in the non-algebraic case. The first two results rely on the key lemma of \S~\ref{S:lemma}, which brings unity to the present work.

\paragraph{Our first result.}

Zamour's proof of Fact~\ref{f:Zamour} can be greatly simplified. Moreover, we remove an unnecessary assumption.

\begin{theorem}\label{t:CoBo}
Let $(C < G)$ be a definable, connected, quasi-Frobenius pair of finite Morley rank. Suppose that $G$ is $U_2^\perp$ with involutions. Suppose further that $C$ is soluble but $G$ is not.
Then the configuration is dihedral and:
\begin{itemize}
\item
either $C$ is a Borel subgroup (`\emph{CoBo}$_2$');
\item
or $(C < G) \simeq (\bK^\times < \PGL_2(\bK))$.
\end{itemize}
\end{theorem}

`CoBo' means `the \textbf{Co}mplement is a \textbf{Bo}rel subgroup', by analogy with the classical pathologies `CiBo' discussed above.
Zamour makes claims about a possible inductive strategy for outright elimination of CoBo$_2$, which we deem bold. He suggests proving that $C$ is properly contained in a Borel subgroup.
But as predicted by \cite{CJTame, DGroupes, DJInvolutive}, the group $C$ (if soluble) should be maximal as such, viz.~a Borel subgroup. This has been known since \cite{CJTame} though proved by different methods in different settings.
Zamour's `weaker' assumption that all Borel subgroups are generic is already known to contradict $N_\circ^\circ$-ness \cite[Proposition~4.1.35]{DGroupes}. In short if Zamour's purported strategy worked, it would eliminate CiBo$_2$ by a naive argument and we do not believe there is a naive one. (We would love to be proved wrong and offer cheese in reward.)

\paragraph{Our second result.}

Although quasi-Frobenius pairs are mostly interesting in the dihedral case \emph{via} the $A_1$ conjecture, we also radicalise strongly embedded configurations. We prove the following.


\begin{theorem}\label{t:1or2}
Let $(C < G)$ be a definable, connected, quasi-Frobenius pair of finite Morley rank. Suppose that $G$ is $U_2^\perp$ with involutions.
Then the index $[N_G(C):C]$ is either $1$ (`strongly embedded') or $2$ (`dihedral'). In either case the Pr\"ufer $2$-rank is $1$ and $C = C_G^\circ(i)$ for its unique involution.
\end{theorem}

Hence, if $G$ is $U_2^\perp$ with involutions, then `quasi-Frobenius, strongly embedded' reduces to `Frobenius'.

\begin{remarks*}\leavevmode
\begin{itemize}
\item
We do not claim that Sylow $2$-subgroups are connected; they could be as in $\SL_2(\bK)$ in characteristic not $2$. A highly pathological Frobenius pair $(C < G)$ with $C \simeq \SL_2(\bK)$ has been haunting researchers in the topic, and has not yet vanished.
\item
We do not exclude the possibility of quasi-Frobenius pairs with no involutions and $N_G(C) > C$. (One would expect to use \cite[Theorem~5]{BCSemisimple} to go any further, which is wide open.)
\item
Elimination of \cite{DJInvolutive}'s CiBo$_1$, even assuming disjointness of $C_G(i)$, is a notouriously challenging open problem. In the quasi-Frobenius setting, things could even be worse with a non-soluble $C$.
\end{itemize}
\end{remarks*}


\paragraph{Our third result.}

The following is absent from \cite{DWGeometry} and \cite{ZQuasi}.

\begin{theorem}[{\cite[Corollaire~4.1.31]{DGroupes}; \cite[Corollaire~3.27]{DGroupes2}}]\label{t:G3C}
Let $(C < G)$ be a definable, connected, quasi-Frobenius pair of finite Morley rank. Suppose that $G$ is $U_2^\perp$ with involutions, and the configuration is dihedral. Then $\dim G = 3 \dim C$.
\end{theorem}

\begin{remarks*}\leavevmode
\begin{itemize}
\item
This is of course unambitious compared to the $A_1$ conjecture, which would immediately imply Theorem~\ref{t:G3C}.
\item
There is no dimension estimate in the strongly embedded case. (See the proof.)
\end{itemize}
\end{remarks*}


\paragraph{Loose ends.}
We list issues not resolved by the present work.
\begin{enumerate}
\item
Definably linear, quasi-Frobenius pairs. Zamour has some unpublished material in his PhD.
\item
Study of quasi-Frobenius pairs if there are no involutions; in particular proving that $N = C$. (This did not seem of utmost interest to all three authors.)
\item
Solubility in the strongly embedded case. This is likely to be very hard as one may add strong assumptions such as $N_\circ^\circ$ and still get a hard problem \cite{DJInvolutive}.
\item
The $A_1$ conjecture. Under exploration by Wiscons and D.
\end{enumerate}

Moreover, Theorem~\ref{t:1or2} motivates the following question in finite group theory.

\begin{question*}
What can be said about quasi-Frobenius pairs where $G$ is finite and $[N_G(H):H] = 2$? The question is asked both for soluble and non-soluble $G$; both with and without the \textsc{cfsg}; both with and without character theory.
\end{question*}

\paragraph{Acknowledgements.}
The key lemma and Theorem~\ref{t:CoBo} were proved in February 2022 during an \textsc{ecos}-Nord visit of D.~to C.~in Bogot\'{a}, D.~C. We warmfully thank the \textsc{ecos}-Nord programme, and fromagerie Laurent Dubois. Theorem~\ref{t:1or2} was proved in September-October 2022 during a `\textsc{cnrs} d\'{e}l\'{e}gation' of D.~to visit A.~in Lyon.
Theorem~\ref{t:G3C} was added when C.~came to Lyon in April-May 2023 to visit A.~and D.
D.~heartfully thanks Bruno Poizat and Ulla Karhumäki for hospitality and conversation.

\section{The Key Lemma}\label{S:lemma}

The proofs of Theorems~\ref{t:CoBo} and \ref{t:1or2} rely on a common key lemma, missed in \cite[Th\'{e}or\`{e}me~4.5.5]{ZQuasi}.

\begin{notation*}
Let $\Gamma = \bigcup_{g \in G} C^g$ be the set of \emph{bright} elements and $\Delta = G \setminus \Gamma$ be the set of \emph{dark} elements.
\end{notation*}

Light matter is generic, viz.~$\Gamma \subseteq G$ is a generic inclusion. Combining this with genericity of centralisers of decent tori \cite{CGood}, we find that every toral element is bright. In particular, if $G$ is $U_2^\perp$, then any involution lies in a unique conjugate of $C$.
Dark elements do not give sufficient grip; but dark, strongly real elements will. They exist by \cite[Theorem~A]{DWGeometry}.

\begin{lemma*}
Let $(C < G)$ be a definable, connected, quasi-Frobenius pair of finite Morley rank. Suppose that $G$ is $U_2^\perp$ with involutions. Let $d \in G$ be dark and inverted by an involution $i$. Then $A = C_G^\circ(d)$ is dark, abelian and inverted by $i$. Moreover, for $a \in C_G^\circ(d) \setminus \{1\}$, one has $C_G^\circ(d) = C_G^\circ(a)$.
\end{lemma*}
\begin{proof}
Say $d = ij$ with $j$ another involution, and let $A = C_G^\circ(d)$.
Let $C_i$ be the unique conjugate of $C$ containing $i$, and define $C_j$ likewise.
Then $i$ normalises $A$, and $C_A^\circ(i) = C_G^\circ(i, j) \leq C_i \cap C_j$. If $C_i = C_j$ then $d = ij \in C_i$ is bright, a contradiction. Thus $C_i \neq C_j$ and $C_A^\circ(i) = 1$; therefore $i$ inverts $A$ (which is abelian). So does $j$.

Suppose some $c \in A$ is bright. Up to conjugacy we may suppose $c \in C$; now $i$ and $j$ invert $c$ so they normalise $C$. Let $N = N_G(C)$.
If $[N:C]$ is odd, then $i, j, d \in C$: a contradiction. So by Fact~\ref{f:DW}, the configuration is dihedral and in particular $[N:C] = 2$. Now exactly one of $i, j$ is in $C$ and the other in $N\setminus C$. But the involution in $C$ is unique so $d = ij$ is an involution, a contradiction again. Therefore $A$ consists of dark elements.

Finally let $a \in A\setminus\{1\}$. We have proved that $a$ is dark and inverted by $i$. So the above applies: $C_G^\circ(a)$ is inverted by $i$, and $j$, hence centralised by $d$. Thus $C_G^\circ(a) \leq C_G^\circ(d) \leq C_G^\circ(a)$ by abelianity. We are done.
\end{proof}
\begin{remarks*}\leavevmode
\begin{itemize}
\item
One easily proves that $A$ is \textsc{ti}, and maximal as a definable, connected, nilpotent group.
\item
A naive Jordan decomposition could be deduced; it is implicit in the identification process of Theorem~\ref{t:CoBo}.
\item
By \textsc{ti}-ness and non-genericity of dark matter, $A$ is \emph{not} almost self-normalising in $G$.
However Borel subgroups containing $A$ will tend to be non-generic.
See our discussion of Theorem~\ref{t:CoBo} in the introduction.
\end{itemize}
\end{remarks*}

\section{The proofs}\label{S:proofs}

All proofs are independent; the first two rely on the key lemma.

\subsection{Proof of Theorem~\ref{t:CoBo}}\label{S:CoBo}

A possible proof is to invoke Theorem~\ref{t:1or2}, then rely on two subcases: Fact~\ref{f:Zamour} and \cite[Th\'{e}or\`{e}me~1.0.7]{ZQuasi}. This is artificial and convoluted; moreover we found Theorem~\ref{t:1or2} only afterwards, as a by-product of the key lemma.

Instead our proof \emph{simplifies} Zamour's.
It still uses some material from \cite{ZQuasi}, namely the following three facts which generalise their `Frobenius' counterparts by Nesin.

\begin{fact}[{induced structure: \cite[Lemme~4.2.1]{ZQuasi}, extending~\cite[Lemma~11.10]{BNGroups}}]\label{f:conjugacy}
Let $(C < G)$ be a definable, connected quasi-Frobenius pair of finite Morley rank. Let $H < G$ be definable and connected.
If $1 < H \cap C < C$ then $(H \cap C < H)$ is another definable, connected, quasi-Frobenius pair.
If $C_1, C_2$ are $G$-conjugates of $C$ with $H \cap C_i \neq 1$ then $C_1$ and $C_2$ are $H$-conjugate.
\end{fact}

\begin{fact}[{non-simple splitting: \cite[Proposition~4.2.3]{ZQuasi}, extending~\cite[Lemma~11.21]{BNGroups}}]\label{f:nonsimple}
Let $(C < G)$ be a definable, connected quasi-Frobenius pair of finite Morley rank. Suppose that $C$ has an involution, and there is an infinite, definable, connected, normal $A \triangleleft G$ avoiding $C$. Then $G$ splits with an abelian kernel and complement $C$.
\end{fact}

\begin{fact}[{soluble analysis: \cite[Lemme~4.3.1]{ZQuasi}, extending~\cite[Theorem~11.32]{BNGroups}}]\label{f:soluble}
Let $(C < G)$ be a definable, connected quasi-Frobenius pair of finite Morley rank. Suppose that $G$ is soluble. Then $(C < G)$ is a Frobenius pair; moreover, $G$ splits as $G = G'\rtimes C$ and $C$ is abelian. Last, $\bigcup_{g \in G} C^g$ covers $G\setminus G'$.
\end{fact}

The proof of Theorem~\ref{t:CoBo} starts here.

\begin{notation*}\leavevmode
\begin{itemize}
\item
Let $(C < G)$ be a definable, connected, quasi-Frobenius pair of finite Morley rank. Suppose that $G$ is $U_2^\perp$ with involutions.
\item
Suppose that $C$ is soluble but $G$ is not.
\item
Fix a Borel subgroup $B \geq C$.
\item
Suppose $B > C$. 
This implies at once that $B$ is not abelian.
\item
Let $i$ be an involution.
Following Bender, for $k \in i^G$ let $T_k = \{b \in B: b^k = b^{-1}\}$.
(Here $B$ remains implicit in the notation.)
\end{itemize}
\end{notation*}

\begin{proposition}[{cf.~\cite[Lemme~4.6.3]{ZQuasi}}]\label{p:genericinvolutions}\leavevmode
\begin{enumerate}[label=(\roman*), series=claims]
\item\label{i:generick}
$\{k \in i^G\setminus N_G(B): \dim T_k \geq \dim B - \dim C_G(i)\}$
is generic in $i^G$.
\end{enumerate}
\end{proposition}
\begin{remarks*}\leavevmode
\begin{itemize}
\item
One needs only solubility of $N_G^\circ(B)$ so far. The inclusion $C < B$ is used from~\ref{i:B'dark} on, and the extra clause that $B$ is a Borel subgroup from~\ref{i:NB'=B} on.
\item
Actually solubility of $N_G^\circ(B)$ can be \emph{proved} using $C < B$ and solubility of $B$. This is done in~\ref{i:NB'=B}.
\item
The case where $C = C_G^\circ(i)$ is a Borel subgroup yields no information since $T_k$-sets are finite \cite{CJTame, DGroupes2, DJInvolutive}.
\end{itemize}
\end{remarks*}

\begin{proof}\leavevmode
\begin{enumerate}[label={\itshape (\roman*)}, series=proofs]
\item
As a contradiction using $\deg i^G = 1$, suppose $i^G \cap N_G(B)$ generic in $i^G$. For $g \in G$ let $X_g = i^G \cap N_G(B^g) = (i^G \cap N_G(B))^g$, a generic subset of $i^G$. By the \textsc{dcc}, there are $g_1, \dots, g_n$ with:
\[L = \bigcap_{g \in G} N_G(B)^g = \bigcap_{r = 1}^n N_G(B)^{g_r}.\]
Now $Y = X_{g_1} \cap \dots \cap X_{g_n} \subseteq L$, and $Y$ is generic in $i^G$. So $K = L^\circ$ is a non-trivial, definable, connected, normal subgroup. By solubility of $N_G^\circ(B)$, $K$ is soluble.

If $K \cap C = 1$, then the configuration splits by non-simplicity (Fact~\ref{f:nonsimple}). Since $K$ and $C$ are soluble, so is $G$: a contradiction. So $K \cap C \neq 1$ and $(K\cap C < K)$ is a quasi-Frobenius pair by Fact~\ref{f:conjugacy}. But $K \trianglelefteq G$. By a Frattini argument and conjugacy in the induced structure (Fact~\ref{f:conjugacy}), $G = K\cdot N_G(C) = K \cdot N_G^\circ(C) = K \cdot C$. Since both are soluble, so is $G$: a contradiction. So $i^G \setminus N_G(B)$ is generic in $i^G$.

The rest is the Bender computation \cite[Proposition~2]{DJInvolutive}, spreading $i^G$ in translates of $B$.
\qedhere
\end{enumerate}
\end{proof}




Since $(C < B)$ is a quasi-Frobenius pair by Fact~\ref{f:conjugacy}, it bears its own notion of darkness. \ref{i:B'dark} says it is induced by that of $(C < G)$.

\begin{proposition}\label{p:intersectioncontrol}\leavevmode
\begin{enumerate}[resume*=claims]
\item\label{i:B'dark}
$B' \neq 1$ is the set of $G$-dark elements of $B$;
\item\label{i:NB'=B}
$N_G^\circ(B') = B$;
\item\label{i:preBruhat}
for $g \notin N_G(B)$ one has $B' \cap B^g = 1$. In particular, $B \cap B^g$ is abelian.
\end{enumerate}
\end{proposition}
\begin{proof}\leavevmode
\begin{enumerate}[resume*=proofs]
\item
Recall $B' \neq 1$ since otherwise $B \leq N_G^\circ(C) = C$.
By soluble splitting (Fact~\ref{f:soluble}), $B = B'\rtimes C$ and $C$ is abelian. For $i$ an involution in $C$, one gets $C_{B'}^\circ(i) = 1$ so $i$ inverts $B'$. Hence $B' = [B, i]$ is abelian.
By Fact~\ref{f:soluble}, $B \cap \Delta \subseteq B'$. But conversely if $x \in B'\cap C^g \setminus\{1\}$ for some $G$-conjugate of $C$, then $B' \leq C_G^\circ(x) \leq C^g$. Therefore $(B\cap C^g < B)$ is a quasi-Frobenius pair, so $B\cap C^g$ is a $B$-conjugate of $C$ by Fact~\ref{f:conjugacy}. Thus $C^g \leq B$, and $C^g$ is a $B$-conjugate of $C$. Now $x \in B'\cap C^g$ is a contradiction. This proves $B' \subseteq \Delta$.

\item
Let $H = N_G^\circ(B') \geq B$. Since $C < B \leq H$, we find that $(C < H)$ is a quasi-Frobenius pair by Fact~\ref{f:conjugacy}. But $1 < B' \triangleleft H$ so by non-simple splitting (Fact~\ref{f:nonsimple}), $H$ splits with an abelian kernel and $C$ as a complement. Hence $H$ is soluble (this does not use~\ref{i:generick} but implies solubility of $N_G^\circ(B) \leq H$). Finally using the definition of a Borel subgroup, $H = B$.
\item
Keep~\ref{i:B'dark} in mind.
Let $d \in (B' \cap B^g)\setminus\{1\}$. Then $d \in B'$ is dark, so $d \in B^g \cap \Delta = (B^g)'$. Then $B', (B^g)' \leq C_G^\circ(d)$ which is abelian by the key lemma; now $(B^g)'\leq C_G^\circ(B')$. But by~\ref{i:NB'=B}, the key lemma and the inner structure of $B$ one has $C_G^\circ(B') \leq C_B^\circ(B') \subseteq \Delta \cap N_G^\circ(B') \subseteq \Delta \cap B = B'$, and therefore $(B^g)' = B'$. Take connected normalisers and apply~\ref{i:NB'=B} again to find $B^g = B$, a contradiction. Therefore $(B \cap B^g)' \leq B' \cap B^g = 1$, and $B \cap B^g$ is abelian.
\qedhere
\end{enumerate}
\end{proof}

So easily obtained an intersection control emphasises the strength of the `quasi-Frobenius' assumption
: it entirely bypasses the local analysis of \cite{DJInvolutive}.

\begin{proposition}[{cf.~\cite[Lemme~4.6.5]{ZQuasi}}]\leavevmode
\begin{enumerate}[resume*=claims]
\item\label{i:Tk=C}
For $k$ generic in $i^G$, the set $T_k$ equals $(B\cap B^k)^\circ$ and is a conjugate of $C$;
\item\label{i:dihedral}
the configuration is dihedral.
\end{enumerate}
\end{proposition}
\begin{proof}\leavevmode
\begin{enumerate}[resume*=proofs]
\item
Fix generic $k$. By~\ref{i:generick}, $k$ does not normalise $B$ but $\dim T_k \geq \dim B - \dim C_G(i)$.
By~\ref{i:preBruhat}, $B \cap B^k$ is abelian. So $T_k$ is an abelian group and $T_k \leq B\cap B^k$.

Suppose that $T_k$ contains a dark element $d$; then $d \in B'$ by~\ref{i:B'dark}, and $d$ is strongly real inverted by $k$. By the key lemma, $k$ inverts $C_G^\circ(d) \geq B'$, so $k$ normalises $B'$. Hence $k \in N_G(N_G^\circ(B')) = N_G(B)$ by~\ref{i:NB'=B}, a contradiction. So $T_k$ consists only of bright elements.
Now if $s, t \in T_k$ and $t \in C$, then by abelianity $C^s = C$ and $s \in N_B(C) = C$ by Fact~\ref{f:soluble}. So all elements are in the \emph{same} conjugate; we may assume $T_k \leq C$.

But then $T_k = T_k^k \leq C^k$, so $C = C^k \leq B \cap B^k$.
We conclude. Let $C_k$ be the conjugate of $C$ containing $k$. If $X = B\cap C_k \neq 1$, then both $(X < B)$ and $(C < B)$ are quasi-Frobenius pairs. By conjugacy in the induced structure (Fact~\ref{f:conjugacy}), we find $\dim X = \dim C$, so $C_k \leq B$ and $k \in C_k \leq B$, a contradiction.
Therefore $X = 1$ and $C_{B\cap B^k}^\circ(k) \leq B \cap C_k = 1$, so 
$k$ inverts $(B\cap B^k)^\circ$. This means $(B\cap B^k)^\circ \leq T_k$.
Summing up,
\[T_k \leq C \leq (B\cap B^k)^\circ \leq T_k,\]
as wanted.
\item
Now $k \in N_G(C) \setminus C$: the configuration is dihedral by Fact~\ref{f:DW}.
\qedhere
\end{enumerate}
\end{proof}

\begin{remarks*}\leavevmode
\begin{itemize}
\item
\ref{i:dihedral} is the reason why Zamour's assumption on $[N:C]$ is unnecessary.
\item
Mind the connected component when proving that $k$ inverts $(B\cap B^k)^\circ$. One may not claim $C_G(k) \leq C_k$, as it is false in the dihedral case. 
\item
One could prove $T_k = B\cap B^k$, but this is not used.
\end{itemize}
\end{remarks*}

\begin{proposition}\leavevmode
\begin{enumerate}[resume*=claims]
\item
$G \simeq \PGL_2(\bK)$.
\end{enumerate}
\end{proposition}

\begin{proof}\leavevmode
\begin{enumerate}[resume*=proofs]
\item
We first observe $\dim B' \geq \dim C$.
Let $c \in C \setminus\{1\}$ centralise some $d \in B'\setminus\{1\}$. By Fact~\ref{f:soluble}, $d$ normalises $C_B^\circ(c) = C$ and therefore $d \in C$: a contradiction. So the action of $C$ on $B'$ is free, implying $\dim B' \geq \dim C$.
(Using $\dim C = \dim T_k \geq \dim B'$ and linearisation methods we even have equality and can push to $B \simeq \GA_1(\bK)$ for some field $\bK$, but this is not used.)

Thanks to~\ref{i:Tk=C}, to generic $k \in i^G$ associate the conjugate of $C$ equal to $T_k$. The range is contained in $\{C^b: b \in B\} = \{C^b: b \in B'\}$ with same dimension as $B'$. The fibre is easily seen to have dimension exactly that of $C$. Thus:
\[\dim G - 2 \dim C = \dim i^G - \dim C \leq \dim B'.\]

We next show $N_G(B) = B$. By a Frattini argument and Fact~\ref{f:conjugacy}, $N_G(B) = B \cdot (N_G(B) \cap N_G(C))$. Suppose there is $x \in N_G(B) \cap N_G(C) \setminus C$. Since the configuration is dihedral by~\ref{i:dihedral}, Fact~\ref{f:DW} implies that $x$ is an involution inverting $C$. It also normalises $B'$, which consists of dark matter by~\ref{i:B'dark}: so $C_{B'}(x)^\circ = 1$, and $x$ inverts $B'$. So $x$ inverts both $C$ (which is $2$-divisible) and $B'$; therefore $C$ centralises $B'$, a contradiction. So $N_G(B) = B$.

Since for $g \notin B$ one has $B'\cap B^g = 1$ by~\ref{i:preBruhat}, we deduce:
\[\dim G \geq \dim B + \dim B' \geq \dim B + \dim C = \dim B' + 2 \dim C \geq \dim G.\]
This is enough to get the Bruhat decomposition $G = BgB' \sqcup B$ for $g \notin B$ and retrieve the $(B, N)$-pair, like in \cite{ZQuasi} or \cite[Proposition~3]{DJInvolutive}, or any identification theorem for $\PGL_2(\bK)$ known so far in the theory of groups of finite Morley rank.
\qedhere
\end{enumerate}
\end{proof}

\subsection{Proof of Theorem~\ref{t:1or2}}\label{S:1or2}

The proof adapts \cite{BCJMinimal}. It is independent from \S~\ref{S:CoBo} but still uses the key lemma of \S~\ref{S:lemma}.

\begin{notation*}\leavevmode
\begin{itemize}
\item
Let $(C < G)$ be a definable, connected, quasi-Frobenius pair of finite Morley rank. Suppose that $G$ is $U_2^\perp$ with involutions.
\item
Let $N = N_G(C)$. If $[N:C]$ is even then we are done by Fact~\ref{f:DW}. So suppose that $[N:C]$ is odd. By Fact~\ref{f:DW}, $(N < G)$ is strongly embedded.
\item
Let $I$ be the set of involutions in $G$ and $\check{I} = I \setminus N$. Notice that since $[N:C]$ is odd, $\check{I} = I \setminus C$.
By strong embedding, $I$ is a single $G$-conjugacy class, and $I\cap N = I \cap C$ is a single $N$-conjugacy class \cite[Theorem~10.19]{BNGroups}.
\end{itemize}
\end{notation*}


\begin{proposition}[a first generic set]\leavevmode
\begin{enumerate}[label={\itshape (\roman*)},series=claims]
\item\label{i:checkI}
$\check{I}$ is generic in $I$;
\item\label{i:mu}
the product map $\mu\colon C \times \check{I} \to G$ has trivial fibres;
\item\label{i:CIgeneric}
$C\cdot \check{I}$ is generic in $G$;
\item\label{i:CCi}
for any $i \in I \cap C$ one has $C = C_G^\circ(i)$;
\item\label{i:stronglyrealinC}
non-trivial strongly real elements of $C$ are involutions;
\item
it is enough to prove $N = C$.
\end{enumerate}
\end{proposition}
\begin{proof}\leavevmode
\begin{enumerate}[label={\itshape (\roman*)},series=proofs]
\item
Otherwise $I \cap N$ is generic in $I$; by strong embedding the latter is a single $G$-conjugacy class, so it has degree $1$.
Then for $g \in G \setminus N$, one has $(I \cap N) \cap (I\cap N)^g \neq \emptyset$, against strong embedding of $(N < G)$.
\item
If $c_1 k_1 = c_2 k_2$ in obvious notation with $(c_1, k_1) \neq (c_2, k_2)$, then $1 \neq c_1^{-1} c_2 = k_1 k_2$ is an element of $C$ inverted by $k_1$. Since $C$ is \textsc{ti}, it follows $k_1 \in N$, against $k_1 \in \check{I}$.
\item
Let $i \in C$ be an involution; since $C$ is \textsc{ti}, one has $C_G^\circ(i) \leq C$.
The map $\mu$ has trivial fibres by~\ref{i:mu}, so using~\ref{i:checkI}:
\[\dim (C\cdot \check{I}) = \dim (C \times \check{I}) = \dim C + \dim I \geq \dim C_G^\circ(i) + \dim G - \dim C_G^\circ(i) = \dim G,\]
proving genericity.
\item
In the proof of~\ref{i:CIgeneric} we also have equality in $\dim C_G^\circ(i) \leq \dim C$, so by connectedness of the latter, $C = C_G^\circ(i)$.
\item
Let $1 \neq k \ell \in C$ be a product of two involutions. Then $k$ inverts $k\ell$ so it normalises $C$; since $[N:C]$ is odd, $k \in C$ and $\ell \in C$ likewise. Now by~\ref{i:CCi} both are central in $C$, so $k\ell$ is an involution too.
\item
Suppose this holds. By strong embedding and~\ref{i:CCi}, $N = C = C_G^\circ(i)$ conjugates its involutions, so $i$ is the only involution in $C$; hence the Pr\"ufer $2$-rank of $G$, which equals that of $C$, is exactly $1$. (One could also rely on \cite[Lemma~11.20]{BNGroups}.)
\qedhere
\end{enumerate}
\end{proof}

We henceforth assume $C < N$ and derive a contradiction.

\begin{notation*}\leavevmode
\begin{itemize}
\item
Let $\sigma \in N \setminus C$.
\item
Also let $I_\sigma = (C \sigma I) \cap \check{I}$.
\end{itemize}
\end{notation*}

\begin{proposition}[a generically dark coset]\leavevmode
\begin{enumerate}[resume*=claims]
\item\label{i:Isigma}
$I_\sigma$ is generic in $I$;
\item\label{i:sigma}
we may suppose that $\sigma$ is strongly real and dark; \emph{we do so from now on};
\item\label{i:CCsigma}
$C^\circ_C(\sigma) = 1$;
\item\label{i:sigmaCgenericCsigma}
$\sigma^C$ is generic in $\sigma C$.
\end{enumerate}
\end{proposition}
\begin{proof}\leavevmode
\begin{enumerate}[resume*=proofs]
\item
By \ref{i:CIgeneric}, $C\cdot \check{I} \subseteq C\cdot I$ is generic in $G$; translating, so is $\sigma C I = C\sigma I$. By connectedness, $(C\sigma I) \cap (C\check{I}) = C I_\sigma$ is generic in $G$ as well. But $I_\sigma \subseteq \check{I}$ so by \ref{i:mu}, one finds:
\[\dim C + \dim I = \dim G = \dim(C \cdot I_\sigma) = \dim C + \dim I_\sigma,\]
which implies $\dim I_\sigma = \dim I$.
\item
It follows from~\ref{i:Isigma} that $I_\sigma \neq \emptyset$. So there is an equation $c \sigma k = \ell$ in obvious notation. Now $c\sigma$ is strongly real; up to considering this element instead of our original $\sigma$, we shall suppose that $\sigma$ is strongly real.
If $\sigma$ is bright, then it lies in a conjugate of $C$.
By~\ref{i:stronglyrealinC}, $\sigma$ is an involution of $N$. But $\sigma \in N$, so $\sigma \in C$: a contradiction.
\item
By~\ref{i:sigma} we may apply the key lemma: $C_G^\circ(\sigma)$ consists of dark matter. In particular, $C_C^\circ(\sigma) = 1$.
\item
Computing modulo $C$ one has the inclusion $\sigma^C \subseteq \sigma C$.
By~\ref{i:CCsigma} one has $C_C^\circ(\sigma) = 1$, so $\dim(\sigma^C) = \dim C$. Hence $\sigma^C$ is a generic subset of $\sigma C$.
\qedhere
\end{enumerate}
\end{proof}

\begin{notation*}
Let $\check{Z} = C_G(\sigma) \setminus N$.
\end{notation*}

\begin{proposition}[a second generic set and contradiction]\leavevmode
\begin{enumerate}[resume*=claims]
\item\label{i:ZgenCsigma}
$\check{Z}$ is generic in $C_G(\sigma)$;
\item\label{i:CZC:1}
$\dim (C \check{Z} C) = 2 \dim C + \dim C_G(\sigma)$;
\item\label{i:fwelldefined}
Every $k \in I_\sigma$ inverts a unique element of $\sigma C$;
\item\label{i:CZC:2}
$\dim I \leq \dim C + \dim C_G(\sigma)$;
\item\label{i:CZCgeneric}
$C \check{Z} C$ is generic in $G$;
\item
contradiction.
\end{enumerate}
\end{proposition}
\begin{proof}\leavevmode
\begin{enumerate}[resume*=proofs]
\item
One has $\dim (C_G(\sigma) \cap N) = \dim C_N^\circ(\sigma) = \dim C_C^\circ(\sigma) = 0$ by~\ref{i:CCsigma}, so $\check{Z}$ is generic in $C_G(\sigma)$.
\item
We shall prove that the product map $C \times \check{Z} \times C \to G$ has finite fibres. Consider equations of the form $c_1 z c_2 = c'_1 z' c'_2$, in obvious notation. Since $C$ is a subgroup, we may assume $c'_1 = c'_2 = 1$, so this reduces to treating equations of the form $c_1 z c_2 = z'$.
Applying $\sigma$ yields:
\[c_1^\sigma z c_2^\sigma = z'^\sigma = z' = c_1 z c_2,\]
whence $z^{-1} [c_1, \sigma] z = [c_2^{-1}, \sigma] \in C^z \cap C$. By definition $z \notin N$, so $[c_1, \sigma] = 1$. But this happens only finitely often by~\ref{i:CCsigma}. Likewise for $c_2$.
So fibres are finite and the desired equality follows from~\ref{i:ZgenCsigma}.
\item
By definition, every $k \in I_\sigma$ can be written $k = c \sigma \ell$ with $c \in C$ and $\ell \in I$. Then $k$ inverts $c\sigma \in C\sigma = \sigma C$.
Now suppose that $\sigma c_1, \sigma c_2 \in \sigma C$ are \emph{two} elements inverted by $k$. Let $d = c_1^{-1} c_2$. Then 
\[d^k = (c_1^{-k} \sigma^{-k}) (\sigma^k c_2^k) = (\sigma c_1)^{-k} (\sigma c_2)^k = \sigma c_1\cdot c_2^{-1} \sigma^{-1} \in C^k \cap C^{\sigma^{-1}} = C^k \cap C.\]
Since $k \in I_\sigma \subseteq \check{I}$ does not normalise $C$, we find $d = 1$, as desired.
%
\item
By~\ref{i:fwelldefined}, the following map is well-defined and definable: $f\colon I_\sigma \to \sigma C$ taking $k$ to the only element of $\sigma C$ inverted by $k$. We bound the dimension of its fibres.

First, $f(I_\sigma) \subseteq \sigma^C$. Indeed, let $\tau \in f(I_\sigma)$. Then $\tau$ is strongly real; not being an involution, it is dark by~\ref{i:stronglyrealinC}. Now by \ref{i:sigmaCgenericCsigma}, $\tau^C$ is generic in $\sigma C$, hence it intersects $\sigma^C$: and $\tau$ is a $C$-conjugate of $\sigma$.

Therefore if $f(\ell) = f(k) = \tau$, then $k\ell \in C_G(\tau)$ and $\ell \in k C_G(\tau)$. But $\tau$ and $\sigma$ are conjugate, so:
\[\dim f^{-1}(k) \leq \dim \left (k C_G(\tau)\right) = \dim C_G(\tau) = \dim C_G(\sigma).\]

Finally by~\ref{i:sigmaCgenericCsigma} and~\ref{i:Isigma}:
\[\dim C = \dim \left(\sigma C\right) = \dim \sigma^C \geq \dim f(I_\sigma) \geq \dim I_\sigma - \dim C_G(\sigma) = \dim I - \dim C_G(\sigma).\]
\item
Recall from \ref{i:CCi} that $C = C_G^\circ(i)$, so $\dim G = \dim I + \dim C$.
Using \ref{i:ZgenCsigma}, \ref{i:CZC:1} and \ref{i:CZC:2}, one gets:
\[\dim (C\check{Z}C) = 2 \dim C + \dim C_G(\sigma) \geq \dim C + \dim I = \dim G,\]
as desired.
\item
By~\ref{i:CIgeneric} and~\ref{i:CZCgeneric}, both sets $C \check{I}$ and $C \check{Z} C$ are generic in $G$, which is connected. So said sets are not disjoint, yielding an equation of the form $c k = c_1 z c_2$, in obvious notation. Conjugating then left-translating by elements of $C$, we reduce to $ck = z$ with $c \in C$, $k \in \check{I}$, and $z \in \check{Z}$. Therefore $c^\sigma k^\sigma = ck$, so $[c, \sigma] \in C$ is trivial or strongly real inverted by $k$. Since $k \in \check{I}$ does not normalise $C$, we find that $\sigma$ centralises $c$; it thus centralises $k = c^{-1} z$ as well. So $\sigma \in C_G(k) = C_G^\circ(k)$ by Steinberg's torsion theorem \cite{DSteinberg} and strong embedding.
Now $\sigma$ is bright, violating~\ref{i:sigma}.
\qedhere
\end{enumerate}
\end{proof}

\subsection{Proof of Theorem~\ref{t:G3C}}

Here again the argument is independent from \S~\ref{S:CoBo} and \S~\ref{S:1or2}. We reproduce the proof of \cite[Corollaire~4.1.31]{DGroupes} or \cite[Corollaire~3.27]{DGroupes2}.

\begin{notation*}\leavevmode
\begin{itemize}
\item
Let $(C < G)$ be a definable, connected, quasi-Frobenius pair of finite Morley rank. Suppose $G$ is $U_2^\perp$ with involutions, and the configuration is dihedral.
\item
Let $I$ be the set of involutions, which is a single conjugacy class; the structure of the Sylow $2$-subgroup is known by \cite[Proposition~1]{DWGeometry}.
\end{itemize}
\end{notation*}

\begin{proposition}\leavevmode
\begin{enumerate}[label=(\roman*)]
\item
Let $(i, j)$ be a generic pair of involutions. Then there is a unique involution commuting to both.
\item
$\dim G = 3 \dim C$.
\end{enumerate}
\end{proposition}
\begin{proof}\leavevmode
\begin{enumerate}[label=(\roman*)]
\item
Existence is given by \cite[Proposition~1~(ix)]{DWGeometry}. We prove uniqueness. Suppose $i, j$ are independent involutions and $k \neq \ell$ are two involutions commuting with both $i$ and $j$.

If $k\ell$ is an involution, then $k$ and $\ell$ commute: so $\{1, i, k, \ell\}$ is a four-group, forcing $i = k\ell = j$ likewise, a contradiction. So $k\ell$ is not an involution, and $k\ell \in C_G(i)$ implies $k\ell \in C_G^\circ(i)$. Since $k\ell \in C_G^\circ(j)$ likewise, we find $i = j$, a contradiction again.
\item
The above defines a map $f\colon \bJ \to I$, where $\bJ \subseteq I \times I$ is a generic subset. Now $f$ is clearly surjective by conjugacy. Moreover if $f(i', j') = k$, then $i', j' \in C_G(k)$; conversely, any pair from $C_G(k)\setminus C_G^\circ(k)$ is mapped to $k$. So fibres have dimension exactly $2 \dim C$. Altogether,
\[\dim I = \dim \bJ - 2 \dim C = 2 \dim I - 2 \dim C,\]
so that $\dim G - \dim C = \dim I = 2 \dim C$ and $\dim G = 3 \dim C$.
\qedhere
\end{enumerate}
\end{proof}

\printbibliography

\end{document}